\documentclass[reqno,12pt]{amsart}
\usepackage{a4wide,color,eucal,enumerate,mathrsfs}
\usepackage[normalem]{ulem}
\usepackage{amsmath,amssymb,epsfig,amsthm} 
\usepackage[latin1]{inputenc}
\usepackage{psfrag}
\usepackage{hyperref,cleveref,tikz}

\def\loc{{\mathop\mathrm{\,loc\,}}}

\def\bint{{\ifinner\rlap{\bf\kern.35em--}
\int\else\rlap{\bf\kern.45em--}\int\fi}\ignorespaces}

\def\bbint{{\ifinner\rlap{\bf\kern.35em--}
\hspace{0.078cm}\int\else\rlap{\bf\kern.45em--}\int\fi}\ignorespaces}

\newcommand{\R}{\mathbb R}

\newcommand{\Sph}{\mathbb S}

\newcommand{\divg}{\operatorname{div}}

\newtheorem{thm}{Theorem}[section]
\newtheorem{lem}[thm]{Lemma}
\newtheorem{prop}[thm]{Proposition}
\newtheorem{cor}[thm]{Corollary}

\newtheorem{assu}[thm]{Assumption}

\numberwithin{equation}{section}

\theoremstyle{remark}
\newtheorem{rem}[thm]{Remark}

\def\bint{{\ifinner\rlap{\bf\kern.35em--}
\int\else\rlap{\bf\kern.45em--}\int\fi}\ignorespaces}

\usepackage{graphicx}
\usepackage{import}
\usepackage{xifthen}
\usepackage{pdfpages}
\usepackage{transparent}

\title[A Pohozaev-neck proof of a critical p-Laplacian Harnack inequality]{A Pohozaev-type neck proof of a conditional Harnack inequality in the critical $p$-Laplacian setting}
\author{Guolin Qin, Yi Ru-Ya Zhang}
\date{\today}

\address{State Key Laboratory of Mathematical Sciences, Academy of Mathematics and Systems Science, Chinese Academy of Sciences, Beijing 100190, China}
\email{qinguolin18@mails.ucas.ac.cn}

\address{Institute of Mathematics, Academy of Mathematics and Systems Science, the Chinese Academy of Sciences, Beijing 100190, China}

\email{yzhang@amss.ac.cn}
 \thanks{Both of the authors are funded by the National Key R\&D Program of China  Grant No. 2025YFA1018400. The second author is also funded by  the National Key R\&D Program of China  Grant  No. 2021YFA1003100, NSFC Grant No. 12288201 \& No. 12571128, the Chinese Academy of Sciences, and CAS Project for Young Scientists in Basic Research, Grant No. YSBR-031. }


\begin{document}

\begin{abstract}
We prove a conditional Schoen-type Harnack inequality for positive weak solutions of the critical $p$-Laplace equation
$$
        -\Delta_p u=g(u),\qquad 1<p<n,
$$
under a global critical Sobolev growth assumption and the monotonicity condition that
$s^{-(p^*-1)}g(s)$ is nonincreasing. The result is conditional on two inputs,  the classification of bounded positive entire blow-up limits as Aubin--Talenti $p$-bubbles and a preliminary singular-rate upper control on the normalized necks. Under these two hypotheses, solutions in $B_{3R}$ satisfy
$$
 \Big(\sup_{B_R}u\Big)\Big(\inf_{B_{2R}}u\Big)^{p-1}
 \le C R^{p-n}.
$$
The main point is a Pohozaev-neck argument which upgrades the preliminary singular decay rate $|x|^{-(n-p)/p}$ to the sharp $p$-harmonic fundamental rate $|x|^{-(n-p)/(p-1)}$. The argument replaces the Kelvin-transform and moving-sphere methods available in the conformally invariant semilinear case $p=2$, but unavailable for the general $p$-Laplacian.
\end{abstract}


\maketitle

\section{Introduction}\label{sec:intro}
We study positive weak solutions of the following  quasilinear critical equation with general nonlinearity
\begin{equation}\label{eq:intro-pde}
-\Delta_p u = g(u) \qquad \text{in } \Omega\subset\mathbb{R}^n,
\end{equation}
where $1<p<n$, $\Delta_p u = \operatorname{div}(|\nabla u|^{p-2}\nabla u)$ denotes the $p$-Laplacian, and the nonlinearity $g:(0,\infty)\to(0,\infty)$ satisfies the structural conditions stated in Assumption \ref{ass:g} below.
Throughout the manuscript we write
$$
p^\ast:=\frac{np}{n-p},\qquad q:=p^\ast-1,
\qquad \alpha:=\frac{n-p}{p-1},\qquad a:=\frac{n-p}{p},
$$
and
$$
G(s):=\int_0^s g(t)\,dt.
$$

The qualitative theory of critical elliptic equations traces back to the fundamental work on the Yamabe problem. For the semilinear case $p=2$, the complete classification of  entire solutions of
\begin{equation}\label{eq: intro-semilinear}
    -\Delta u =u^{\frac{n+2}{n-2}}, \qquad u>0, \qquad \text{in }\mathbb{R}^n,
\end{equation}
was obtained by Caffarelli, Gidas and Spruck~\cite{CGS1989}. Subsequently, a simplified proof was proposed by Chen and Li~\cite{CL1991}. Earlier contributions under stronger decay assumptions at infinity are due to Obata~\cite{O1971} and to Gidas, Ni and Nirenberg~\cite{GNN1979}. 

Building on this Liouville-type theorem, Schoen~\cite{S1988} (see also~\cite{S1995}) discovered a striking Harnack inequality for positive solutions of $-\Delta u = u^{\frac{n+2}{n-2}}$ in a ball $B_{3R}$:
\begin{equation}\label{eq:intro-schoen-harnack}
\Bigl(\sup_{B_R}u\Bigr)\Bigl(\inf_{B_{2R}}u\Bigr) \le C(n) R^{2-n}.
\end{equation}
This estimate, together with its associated energy bound, became a cornerstone of the blow-up analysis for critical equations: it provides the basic compactness that underlies the pointwise description of singular solutions, the finite-dimensional reduction arguments, and the refined asymptotics of concentrating sequences; see~\cite{CL1995, CL1997,CL1998,Li1995,Li1996} and the references therein. Li and Zhang~\cite{LZ2003} later extended \eqref{eq:intro-schoen-harnack} to positive solutions of equations $-\Delta u = g(u)$ with a much wider class of nonlinearities, assuming only that the function $s\mapsto s^{-\frac{n+2}{n-2}}g(s)$ is nonincreasing on $(0,\infty)$ and admits a finite positive limit as $s\to\infty$. Schoen's Harnack inequality has been extended to fractional equations
by Caffarelli, Jin, Sire and Xiong~\cite{CJSX2014},  to integral
equations  by Jin and Xiong~\cite{JX2021}, and to conformally invariant fully nonlinear equations by Li and Li~\cite{LL2003, LL2005}. Their proof, as well as the related classification work in~\cite{CGS1989,CL1991}, rests on   the moving sphere (or moving plane) method, which employs the Kelvin transform to compare the values of a solution on different scales, exploiting the conformal invariance of the equation. 

However,  the Kelvin transform and hence the moving sphere method are not currently available when the Laplacian is replaced by the $p$-Laplacian with $p\neq 2$, since the equation is no longer conformally invariant.  For $p\ne2$, the absence of a Kelvin transform is a genuine obstruction in the global classification problem.  

The purpose of this paper is to establish, for all $1<p<n$, a conditional Harnack inequality under two explicit hypotheses: the bounded-entire classification hypothesis for possible blow-up limits, and a preliminary singular-rate control on normalized necks. The main new point is a \emph{Pohozaev-neck} argument which upgrades the preliminary neck control to the sharp outer-sphere infimum estimate needed for the sup$\times$inf Harnack inequality. This replaces the Kelvin transform and moving-sphere method available in the semilinear conformally invariant case $p=2$.
This idea is usually applied in geometric analysis to analyze the "bubbling" phenomenon in sequences of maps; see e.g.~\cite{CM2008, DRS2026}.

We assume the nonlinearity $g$ satisfies the following structural conditions.

\begin{assu}\label{ass:g}
The function $g:(0,\infty)\to(0,\infty)$ is continuous and satisfies the following three conditions.
\begin{enumerate}[(G1)]
\item There exists a constant $\Lambda>0$ such that
\begin{equation}\label{eq:g-upper-critical}
   0<g(s)\le \Lambda s^q\qquad\hbox{for every }s>0.
\end{equation}
This is the global critical upper growth assumption  throughout the paper.
\item The function
$$
   s\longmapsto s^{-q}g(s)
$$
is nonincreasing on $(0,\infty)$.
\item There exists a constant $L\in(0,\infty)$ such that
\begin{equation}\label{eq:g-limit}
   \lim_{s\to\infty}s^{-q}g(s)=L.
\end{equation}
\end{enumerate}
\end{assu}

The model case is $g(s)=s^{p^*-1}$, the Euler--Lagrange equation for the sharp Sobolev embedding $W^{1,p}(\mathbb{R}^n)\hookrightarrow L^{p^*}(\mathbb{R}^n)$.

The second hypothesis concerns the possible entire blow-up limits.  For every $A\in (0,\,\infty]$, define for $s>0$,
\begin{equation}\label{def-ga}
g_A(s):=
\begin{cases}
A^{-q}g(As), & A<\infty,\\
Ls^q, & A=\infty.
\end{cases}
\end{equation}
\begin{assu}[Entire bounded critical classification]\label{ass:classification}
Every bounded positive entire weak solution $V$ of
\begin{equation}\label{eq:critical-entire}
   -\Delta_p V = g_A(V) \qquad\text{in }\R^n
\end{equation}
is an Aubin--Talenti $p$-bubble.

For every sequence $\Gamma_j\to\infty$ and every sequence $v_j$ arising from the doubling normalization, satisfying the normalization, boundedness, critical growth, convergence-to-bubble, and Pohozaev-sign hypotheses (defined explicitly in the statement of Lemma~\ref{lem:neck}), there exist constants $C_s>0$ and $R_s>1$, independent of $j$, such that
\begin{equation}\label{eq:preliminary-estimate}
v_j(y)\le C_s |y|^{-a}  \qquad\text{for } R_s\le |y|\le 4\Gamma_j,
\end{equation}
where $a=\frac{n-p}{p}$.
\end{assu}

The second assumption above  is essential for our analysis. Indeed, if it were to fail, one could then extract, after a suitable rescaling around a distant point, another globally bounded entire critical profile. In this sense, the argument constitutes the local-neck counterpart of the preliminary non-sharp decay estimates that are employed in the finite-energy critical $p$-Laplace setting. In the finite-energy regime, such preliminary estimates can be established by means of a doubling argument combined with small tail-energy bounds, as carried out in \cite{V2016}.  

A closely related sharp condition also appears in the recent work of Ciraolo--Gatti~\cite{CG2025}.  For the pure critical equation
$-\Delta_p u=u^{p^*-1}$ in $\mathbb R^n$, they assume that the infimum over large balls has the fundamental $p$-harmonic decay rate,
namely
$$
\inf_{B_R}u\le C R^{-\frac{n-p}{p-1}},       \qquad R\ge1,
$$
and use this hypothesis to obtain global integrability estimates and, under
boundedness or $L^q$-integrability with $q\ge p^*$, classification as a
$p$-bubble.  In the present paper the analogous sharp estimate is not assumed
on the original solution; rather, the Pohozaev neck lemma upgrades the
preliminary normalized neck control $v_j\lesssim |y|^{-(n-p)/p}$ to the sharp
outer-sphere estimate
$$
 \inf_{\partial B_{\Gamma_j}}v_j   \lesssim \Gamma_j^{-\frac{n-p}{p-1}}.
$$ Thus our neck argument can be viewed as a local blow-up analogue of the sharp infimum behavior used by Ciraolo--Gatti.

\begin{thm}[Harnack inequality]\label{thm:main-intro}
Suppose that $g$ satisfies Assumption~\ref{ass:g} and that Assumption~\ref{ass:classification} holds. Then there exists a constant $C$, depending only on the data in the structural and compactness assumptions, such that for any $R>0$ and any positive weak solution $u$ of $-\Delta_p u = g(u)$ in $B_{3R}$ with $\sup_{B_R}u\ge 1$, we have
\begin{equation}\label{eq:intro-harnack-main}
\Bigl(\sup_{B_R}u\Bigr)\Bigl(\inf_{B_{2R}}u\Bigr)^{p-1} \le C\, R^{\,p-n}.
\end{equation}
\end{thm}
\begin{rem}
When $p=2$, the exponent $p-1=1$ and~\eqref{eq:intro-harnack-main} reduces precisely to the classical sup$\times$inf form~\eqref{eq:intro-schoen-harnack} of Schoen and Li--Zhang.  
\end{rem}

Assumption~\ref{ass:classification} should be understood as a classification hypothesis for the possible entire blow-up limits. It does not claim that the original nonlinearity $g$ is globally a pure power. Indeed, if for some
$A\in(0,\infty]$ a bounded positive entire solution $V$ of
$$
-\Delta_p V=g_A(V)\quad\text{in }\mathbb R^n
$$
is classified as an Aubin--Talenti bubble, then substituting the explicit bubble
profile into the equation shows that $g_A(s)=L_A s^q$ on the range
$$
0<s\le \max_{\mathbb R^n} V
$$
of that particular solution, for some constant $L_A>0$. This  is not a global structural
conclusion on $g$.

Here an Aubin--Talenti $p$-bubble means a function of the form
\begin{equation}\label{eq:bubble-form}
   U_{\lambda,x_0,\mu}(x)
   =A_{n,p,\mu}
   \left(
      \frac{\lambda^{1/(p-1)}}{\lambda^{p/(p-1)}+|x-x_0|^{p/(p-1)}}
   \right)^{(n-p)/p},
\end{equation}
where $\lambda>0$, $x_0\in\R^n$, $\mu>0$, and $A_{n,p,\mu}>0$ is chosen so that
$$
-\Delta_p U_{\lambda,x_0,\mu}
 =  \mu U_{\lambda,x_0,\mu}^{q}  \qquad\text{in }\mathbb R^n .
$$
When $\mu=L$, this is the bubble corresponding to the limiting equation
$g_\infty(s)=Ls^q$. The exact value of $A_{n,p,L}$ is not used. If $U$ is a bubble, then there exists a positive constant $\kappa$ such that
\begin{equation}\label{eq:bubble-tail}
   \lim_{|x|\to\infty}|x|^\alpha U(x)=\kappa.
\end{equation}

In the semilinear case $p=2$, Li and Zhang~\cite{LZ2003} proved exactly this
statement under the monotonicity condition (G2) and local boundedness of $g$; see also \cite{B1997, CL1995} for related classification results under some additional conditions. 
For $p\neq 2$, the pure-power equation \begin{equation}\label{eq:pure-power}
    -\Delta_p V=L V^q
\end{equation}
has been classified in several ranges. 
Catino--Monticelli--Roncoroni explicitly discuss the critical $p$-Laplace classification problem, including
infinite-energy solutions and additional growth hypotheses; see~\cite{CMR2023}. They introduced an integral identity method
and obtained the full classification when $n=2$ or $n=3$ with $3/2<p<2$.
Crucially for the present work, they also proved~\cite[Corollary~1.5]{CMR2023} that
every positive \emph{bounded} weak solution of~\eqref{eq:pure-power}
is an Aubin--Talenti $p$-bubble, provided $n\le 6$, or $n\ge7$ and $p>n/3$.
The finite-energy classification is available from the $D^{1,p}(\R^n)$ theory; see
\cite{DMMS2014,S2016,V2016}. We also mention that, for positive weak solutions,
Ou~\cite{O2025} extended the classification  to $p>(n+1)/3$ for $n\ge3$ without any further assumptions,
and V\'etois~\cite{V2024} further slightly improved the range to $p>p_n$
with $p_n\in(\frac{n}{3},\frac{n+1}{3})$ for $n\ge4$. 
In the recent work \cite{CG2025}, Ciraolo and Gatti classified all positive bounded weak solutions under the assumption that the infimum of the solution behaves properly at infinity.
Ciraolo, Figalli and Roncoroni~\cite{CFR2020} classified positive  $D^{1,p}$-solutions of critical anisotropic $p$-Laplacian equations in convex cones. More recently, \cite{CDGL2026} removed the finite-energy assumption in \cite{CFR2020} in the range $\frac{n+1}{3}<p<n$.
Saintier's blow-up theory provides
related $H^1_p$ decompositions and pointwise estimates in compact-manifold settings; see~\cite{S2006}.

For a general nonlinearity $g$, however, the finite-$A$ blow-up equations
$$
-\Delta_p V=g_A(V),\qquad 0<A<\infty,
$$
need not be pure-power equations a priori. Thus Assumption~\ref{ass:classification}
is retained as an explicit conditional hypothesis. Nevertheless, in the pure-power case $g(s)=Ls^q$, all the rescaled nonlinearities satisfy $g_A(s)=Ls^q$, and the above known classification results  in \cite{CMR2023, O2025, V2024} make the theorem unconditional in their respective parameter ranges.  Outside these ranges, the present theorem reduces the Harnack inequality to two
separate inputs: the classification of bounded entire blow-up limits and the preliminary singular neck control in Assumption~\ref{ass:classification}.

\begin{rem}
    In the limiting case $p=n$, the critical Sobolev embedding is of exponential type, governed by the Moser--Trudinger inequality.  The model equation becomes the $n$-Liouville equation $-\Delta_n u =  e^{u}$. For this problem the appropriate form of the Harnack inequality is replaced by a $\sup + \inf$ estimate.  This type of estimate was first proved for $n=2$ by Shafrir~\cite{S1992}, and was subsequently extended by Brezis, Li and Shafrir~\cite{BLS1993}, Chen and Lin~\cite{CL1998sup, CL1998}, Li~\cite{Li1999}, and Chen and Li~\cite{CL2009}. For general $n\ge 2$, the $\sup + \inf$-type inequality for solutions of $n$-Liouville equation was established by Esposito and Lucia~\cite{EL2024} by using the classification result in \cite{E2018}. 
\end{rem} 

\begin{rem}[Energy-bounded version]
\label{rem:energy-bounded-harnack}
Assume, in addition to the structural hypotheses, that
$$
   g(s)\equiv Ls^q\qquad\text{for all }s>0.
$$
Then there is a useful unconditional variant of Theorem~\ref{thm:main-intro} if one
restricts to solutions with a scale-invariant critical energy bound and power nonlinearity $g(s)=L s^q$. More precisely,
assume that Assumption~\ref{ass:g} holds and let $E>0$. If $u>0$ solves
$$
  -\Delta_p u = L u^q \qquad \text{in } B_{3R},
$$
and
$$
  \int_{B_{3R}} u^{p^\ast}\,dx \le E,
$$
then the same conclusion as in Theorem~\ref{thm:main-intro} holds, with a constant
depending also on $E$:
$$
  \left(\sup_{B_R} u\right)  \left(\inf_{B_{2R}} u\right)^{p-1} \le C(n,p,\Lambda,L,g,E) R^{p-n}.
$$
Indeed, in the contradiction argument one rescales around the doubling-selected point by
$$
 v_j(y)=\widetilde M_j^{-1}   u_j(x_j+\rho_j y),  \qquad
   \rho_j=\widetilde M_j^{-p/(n-p)}.
$$
The $L^{p^\ast}$-energy is invariant under this critical scaling:
$$ 
\int_{B_{4\Gamma_j}} v_j^{p^\ast}\,dy
  =   \int_{B_{4\sigma_j}(x_j)} u_j^{p^\ast}\,dx  \le E.
$$
Consequently, via a  Caccioppoli test, every entire blow-up limit $V$ has finite critical energy. By the
known finite-energy classification for the critical $p$-Laplace equation in
$D^{1,p}(\mathbb R^n)$, see
\cite{DMMS2014,S2016,V2016},
the limit $V$ must be an Aubin--Talenti $p$-bubble. 
Moreover, in the finite-energy regime, the preliminary estimates \eqref{eq:preliminary-estimate} can be established by means of a doubling argument combined with small tail-energy bounds, as carried out in \cite{V2016}.  
Thus the Pohozaev-neck
lemma applies exactly as in the proof of Theorem~\ref{thm:main-intro}.

This observation also explains precisely where the general case is difficult. If the
energy bound is removed, the same blow-up procedure still produces a bounded
positive entire solution of
$$
-\Delta_p V = L V^{p^\ast-1}\qquad \text{in } \mathbb R^n
$$
in the pure-power case, but it need not have finite energy. Hence the finite-energy classification no longer applies. Therefore, in the present approach, removing the energy bound is equivalent to excluding non-bubble bounded entire blow-up limits of infinite energy. This is exactly the content of the bounded-entire classification hypothesis
Assumption~\ref{ass:classification}.
\end{rem}


We now describe the main idea of the proof. The monotonicity of $s^{-q}g(s)$ in (G2) of Assumption~\ref{ass:g}  implies the structural sign condition
$$
   nG(s)-a s g(s)\ge 0,
   \qquad G(s)=\int_0^s g(t)\,dt,\quad a=\frac{n-p}{p}.
$$
Consequently the Pohozaev functional is nonnegative on balls.
However, the opposite sign appears at isolated $p$-harmonic singularities. If
$$
   w(x)=\kappa |x|^{-\alpha}+\beta+o(1),
   \qquad \beta>0,
$$
then the Pohozaev functional has a strictly negative limiting value. This is the basic sign contradiction.

The neck lemma converts this sign contradiction into a quantitative decay statement,
provided a preliminary singular-rate upper control is available on the normalized neck.
Roughly speaking, the preliminary control gives a uniform annular Harnack inequality at
the first-crossing scale; after rescaling, one obtains a $p$-harmonic singular profile whose
Pohozaev sign contradicts the nonnegative sign inherited from the original equation.

The Harnack inequality then follows by contradiction: a failure of the sup$\times$inf
estimate produces a blow-up sequence, the classification hypothesis identifies the core bubble, the second assumption in Assumption~\ref{ass:classification} supplies the preliminary neck control, and the Pohozaev neck lemma upgrades this to the sharp fundamental-rate decay.

The paper is organized as follows. Section~2 fixes notation and recalls elementary facts about the $p$-Laplacian and the Pohozaev functional. Section~3 contains the main analytic ingredients: the Kichenassamy--V\'eron pole theorem, the Pohozaev identity, the sign computation at isolated singularities, and the Pohozaev-neck lemma. Section~4 proves the Harnack inequality by contradiction and blow-up.

\noindent{\bf Acknowledgment}: The second author would like to express his sincere gratitude to Jingang Xiong for bringing this interesting problem to his attention.

\section{Preliminaries: notations and definitions}

Throughout the note, $n$ denotes an integer with $n\ge 2$, and $p$ denotes a real number satisfying $ 1<p<n.$
The Euclidean norm of $x\in\R^n$ is denoted by $|x|$.  If $x_0\in\R^n$ and $r>0$, then
$$
   B_r(x_0):=\{x\in\R^n: |x-x_0|<r\}
$$
is the open Euclidean ball of radius $r$ and center $x_0$.  We write
$B_r:=B_r(0).$
The boundary of $B_r(x_0)$ is denoted by $\partial B_r(x_0)$, and $dS$ denotes Euclidean surface measure
on a sphere.  If the center is the origin, $\nu=x/|x|$ denotes the outward unit normal vector on
$\partial B_r$.  If $u$ is differentiable, then
$$
   u_\nu:=\nabla u\cdot \nu
$$
is its normal derivative on the sphere.

We denote by
$   \omega_{n-1}:=\mathcal H^{n-1}(\Sph^{n-1})$
the surface measure of the unit sphere $\Sph^{n-1}=\partial B_1$.
The critical Sobolev exponent for $W^{1,p}(\R^n)$ is
$ p^*:=\frac{np}{n-p}.$
The exponent appearing in the critical equation is
$$
   q:=p^*-1=\frac{n(p-1)+p}{n-p}.
$$
The decay exponent of the $p$-harmonic fundamental solution is
$$
   \alpha:=\frac{n-p}{p-1}.
$$
Finally, the coefficient which appears in the Pohozaev identity is
$$
   a:=\frac{n-p}{p}.
$$
These constants satisfy the useful identity
\begin{equation}\label{eq:a-alpha-identity}
   a=\frac{\alpha(p-1)}{p}.
\end{equation}

The $p$-Laplacian is written with the sign convention
$$
   \Delta_p u:=\divg\bigl(|\nabla u|^{p-2}\nabla u\bigr),
$$
so that the equation considered below is
$$
   -\Delta_p u=g(u).
$$

A positive constant denoted by $C$ may change from line to line.  Dependence of constants is always
recorded when it is important.  For example, $C=C(n,p,\Lambda)$ means that $C$ depends only on the
listed quantities.

Let $\Omega\subset\R^n$ be open.  A positive weak solution of
\begin{equation*}\label{eq:weak-equation-general}
   -\Delta_p u=g(u)\qquad\hbox{in }\Omega
\end{equation*}
is a function
$$
   u\in W^{1,p}_{\loc}(\Omega),\qquad u>0\quad\hbox{a.e. in }\Omega,
$$
such that $g(u)\in L^1_{\loc}(\Omega)$ and
\begin{equation*}\label{eq:weak-formulation}
   \int_\Omega |\nabla u|^{p-2}\nabla u\cdot \nabla\varphi\,dx
   =\int_\Omega g(u)\varphi\,dx
\end{equation*}
for every  function $\varphi\in C_c^1(\Omega)$.

For any continuous function $h:(0,\infty)\to\R$, we define its primitive by
$$
   H(s):=\int_0^s h(t)\,dt,
$$
whenever the integral is finite.  In particular, for the function $g$ above we write
\begin{equation}\label{eq:F-def}
   G(s):=\int_0^s g(t)\,dt.
\end{equation}

\section{The   pole theorem and the neck lemma}
The proof uses two known or conjectural inputs: the   entire classification Assumption \ref{ass:classification} and the following isolated $p$-harmonic pole theorem of Kichenassamy--V\'eron.

\begin{thm}[{\cite[Theorem 1.1, Remark 1.4]{KV1986}}]\label{thm:KV}
Let $1<p<n$.  Define the $p$-harmonic fundamental profile
\begin{equation}\label{eq:Phi-def}
   \Phi(x):=|x|^{-\alpha}\qquad (x\ne0).
\end{equation}
Let $w$ be a positive $p$-harmonic function in $B_2\setminus\{0\}$,  i.e.\ 
$$
   -\Delta_p w=0\qquad\hbox{weakly in }B_2\setminus\{0\}.
$$
Assume that $w/\Phi$ is bounded near $0$.  Then there exists $\gamma\ge0$  and $\beta\in \mathbb R$ such that
\begin{equation*}\label{eq:KV-bounded-regular-part}
   w-\gamma\Phi\in L^\infty_{\loc}(B_2),\quad w(x)=\gamma \Phi(x) +\beta+o(1)
        \qquad\text{as }x\to0.
\end{equation*}
Moreover, when $\gamma>0$, after shrinking the punctured ball if necessary, the gradient of $w$ does not vanish near the singularity,  and, 
\begin{equation*}\label{eq:KV-gradient-asymptotic}
   \lim_{x\to0}|x|^{\alpha+|\zeta|}D^\zeta(w-\gamma\Phi)(x)=0
\end{equation*}
for every multi-index $\zeta$ with $|\zeta|\ge1$, at points where the derivatives are understood classically.
Moreover,
\begin{equation}\label{eq:KV-delta-mass}
   -\Delta_p w=\gamma^{p-1}c_{n,p}\delta_0
   \qquad\hbox{in }\mathcal D'(B_2),
\end{equation}
where $\delta_0$ is the Dirac mass at $0$, and
\begin{equation*}\label{eq:c-np}
   c_{n,p}:=\omega_{n-1}\alpha^{p-1}.
\end{equation*}
\end{thm}

The theorem above is the normalized version of the main isotropy theorem in
\cite[Theorem 1.1]{KV1986}.  Their fundamental solution differs from \eqref{eq:Phi-def} by a fixed
normalizing constant. More precisely, since in the present paper we use $\Phi(x)=|x|^{-\alpha}$, the coefficient of the Dirac mass becomes $\omega_{n-1}\alpha^{p-1}$. The existence of the finite limit $$\beta=\lim_{x\to0}(w-\gamma\Phi)$$
is obtained in their proof by the same comparison argument which gives the boundedness of the regular part. Serrin's earlier isolated-singularity theory is the precursor of this
input; see~\cite{S1964,S1965}.

\begin{cor}\label{cor:regular-part-gradient-decay}
Let $w>0$ be $p$-harmonic in $B_2\setminus\{0\}$, and suppose that
$$
  w(x)=\gamma |x|^{-\alpha}+\beta+z(x)  \qquad\text{as }x\to0,
$$
where $\gamma\ge0$, $\beta\in\mathbb R$, and $z(x)\to0$ as $x\to0$.
Assume moreover that, when $\gamma>0$, the Kichenassamy--V\'eron weighted derivative asymptotics hold, namely
$$
 |x|^{\alpha+1}|\nabla z(x)|\to0, \qquad  |x|^{\alpha+2}|D^2 z(x)|\to0  \qquad\text{as }x\to0 .
$$
Then
$$
 |x|\,|\nabla z(x)|\to0  \qquad\text{as }x\to0 .
$$   
\end{cor}
\begin{proof}
We first treat the case $\gamma=0$. Then $w=\beta+z$ is bounded near the origin. By the removable singularity theorem for bounded $p$-harmonic functions, $w$ extends as a $p$-harmonic function across $0$. Hence $w\in C^{1,\theta}$ in a smaller ball, for some $\theta\in(0,1)$. Therefore
$$
  |x|\,|\nabla z(x)|  = |x|\,|\nabla w(x)|  \le  |x|\,\|\nabla w\|_{L^\infty(B_{1/2})}
 \to0 .
$$
Thus it remains to consider the case $\gamma>0$.

Set $ h(x):=\gamma |x|^{-\alpha}.$
Then both $w=h+\beta+z$ and $h$ are $p$-harmonic in
$B_2\setminus\{0\}$. We prove the desired estimate by rescaling on fixed
annuli. Let
$$
  A_0:=B_2\setminus \overline{B}_{1/2},  \qquad A_1:=B_{3/2}\setminus \overline{B}_{2/3}.
$$
For $0<r<1/4$, define
$$
  H_r(y):=h(ry)=\gamma r^{-\alpha}|y|^{-\alpha}, \qquad
 Z_r(y):=z(ry),   \qquad y\in A_0 .
$$
Since $z(x)\to0$, and since the weighted derivative asymptotics hold, we have
$$
  \|Z_r\|_{L^\infty(A_0)}\to0,
$$
and
$$
  r^\alpha\|\nabla_y Z_r\|_{L^\infty(A_0)}   +   r^\alpha\|D_y^2 Z_r\|_{L^\infty(A_0)}  \to 0   \qquad\text{as }r\downarrow0 .
$$
On the other hand,
$$
  \nabla_y H_r(y)  = -\alpha\gamma r^{-\alpha}|y|^{-\alpha-2}y.
$$
Hence, on $A_0$,
$$
  c r^{-\alpha}   \le  |\nabla_y H_r(y)|
   \le  C r^{-\alpha},
$$
where $c,C>0$ depend only on $n,p,\gamma$. In particular, for all sufficiently
small $r$,
$$
  |\nabla_y H_r(y)+t\nabla_y Z_r(y)|  \ge c r^{-\alpha}
$$
for every $y\in A_0$ and every $t\in[0,1]$.

Let $\mathcal A(\xi):=|\xi|^{p-2}\xi.$
Since $H_r+\beta+Z_r=w(ry)$ and $H_r=h(ry)$ are both $p$-harmonic with
respect to the $y$-variable, their weak equations give
$$
 \operatorname{div}_y  \bigl( \mathcal A(\nabla_y H_r+\nabla_y Z_r)  -  \mathcal A(\nabla_y H_r)  \bigr)  =0 \qquad\text{in }A_0 .
$$
By the mean value formula,
$$
  \mathcal A(\nabla_y H_r+\nabla_y Z_r)  -  \mathcal A(\nabla_y H_r)  =  M_r(y)\nabla_y Z_r(y),
$$
where
$$
 M_r(y)  :=  \int_0^1  D\mathcal A(\nabla_y H_r(y)+t\nabla_y Z_r(y))\,dt.
$$
Therefore $Z_r$ solves the linear divergence-form equation
$$
 \operatorname{div}_y\bigl(M_r(y)\nabla_y Z_r(y)\bigr)=0
  \qquad\text{in }A_0 .
$$

We now normalize the coefficient matrix by setting
$$
  \mathcal B_r(y):=r^{\alpha(p-2)}M_r(y).
$$
Multiplying the coefficient matrix by a positive constant does not change the
equation, so
$$
  \operatorname{div}_y\bigl(\mathcal B_r(y)\nabla_y  Z_r(y)\bigr)=0   \qquad\text{in }A_0 .
$$
Recall that
$$
   D\mathcal A(\xi)  = |\xi|^{p-2}I  +  (p-2)|\xi|^{p-4}\xi\otimes\xi.
$$
Since $|\nabla_y H_r+t\nabla_y Z_r|$ is comparable to $r^{-\alpha}$ on $A_0,$
the matrices $\mathcal B_r$ are uniformly elliptic  independent of $r$.
Moreover the matrices $\mathcal B_r$ have uniformly bounded $C^1$-norm on $A_0$.
Indeed,
$$
  |D_y^2 H_r(y)|\le C r^{-\alpha}  \qquad\text{on }A_0,
$$
and
$$
  \|D_y^2 Z_r\|_{L^\infty(A_0)}=o(r^{-\alpha}).
$$
Since $D^2\mathcal A(\xi)=O(|\xi|^{p-3})$ away from $\xi=0$, we get $ \|\mathcal B_r\|_{C^1(A_0)}\le C$ 
with $C$ independent of $r$.
Thus, the standard interior $C^1$-estimate for uniformly elliptic linear divergence-form
equations with $C^1$ coefficients therefore gives
$$
   \|\nabla_y Z_r\|_{L^\infty(A_1)} \le   C\|Z_r\|_{L^\infty(A_0)},
$$
where $C$ is independent of $r$. Since $\|Z_r\|_{L^\infty(A_0)}\to 0$, we obtain $  \|\nabla_y Z_r\|_{L^\infty(A_1)}\to 0.$

Finally, let $x\to0$, put $r=|x|$, and set $y=x/|x|\in S^{n-1}\subset A_1$.
Then
$$
  |x|\,|\nabla z(x)|   =   |\nabla_y Z_r(y)|  \le  \|\nabla_y Z_r\|_{L^\infty(A_1)} \to 0 .
$$
This proves the corollary.
\end{proof}

\subsection{The Pohozaev identity and the structural sign}

We first prove an algebraic sign coming from Assumption \ref{ass:g}.

\begin{lem} \label{lem:pohozaev-sign}
Let $g$ satisfy Assumption \ref{ass:g}, and let $G$ be defined by \eqref{eq:F-def}.  Then
\begin{equation}\label{eq:pohozaev-sign}
   nG(s)-a\,s g(s)\ge0\qquad\hbox{for every }s>0.
\end{equation}
\end{lem}

\begin{proof}
Fix $s>0$.  Since $t\mapsto t^{-q}g(t)$ is nonincreasing, for every $t\in(0,s)$ we have
$$
   \frac{g(t)}{t^q}\ge \frac{g(s)}{s^q}.
$$
Therefore
$$
   G(s)=\int_0^s g(t)\,dt
   \ge \frac{g(s)}{s^q}\int_0^s t^q\,dt
   =\frac{s g(s)}{q+1}.
$$
Since $q+1=p^*$, this gives
$$
   G(s)\ge \frac{s g(s)}{p^*}.
$$
Multiplying by $n$ and using
$$
   \frac{n}{p^*}=\frac{n-p}{p}=a,
$$
we obtain \eqref{eq:pohozaev-sign}.
\end{proof}

\begin{prop}[Pohozaev identity]\label{prop:pohozaev}
Let $f:(0,\infty)\to[0,\infty)$ be continuous, and define
$$
   F_f(s):=\int_0^s f(t)\,dt.
$$
Let $u>0$ be a  regular solution of
\begin{equation}\label{eq:pohozaev-eqn}
   -\Delta_p u=f(u)\qquad\hbox{in }B_r.
\end{equation}
Then
\begin{equation}\label{eq:pohozaev-identity}
   P_u(r)=\int_{B_r}\left[nF_f(u)-a\,u f(u)\right]dx,
\end{equation}
where
\begin{equation*}\label{eq:P-u-r}
   P_u(r):=\int_{\partial B_r}
   \left[
       a\,u|\nabla u|^{p-2}u_\nu
       -\frac r p |\nabla u|^p
       +r|\nabla u|^{p-2}u_\nu^2
       +rF_f(u)
   \right]dS.
\end{equation*}
The same identity holds for weak solutions by the standard approximation argument.
\end{prop}

\begin{proof}
Set
$$
   A:=|\nabla u|^{p-2}\nabla u.
$$
The equation \eqref{eq:pohozaev-eqn} is
$$
   -\divg A=f(u).
$$
Multiplying by $x\cdot\nabla u$ and integrating over $B_r$ gives
\begin{equation}\label{eq:pohozaev-step1}
   \int_{B_r} f(u)x\cdot\nabla u\,dx
   =-\int_{B_r}(\divg A)(x\cdot\nabla u)\,dx.
\end{equation}
The right-hand side is
\begin{align}\label{eq:pohozaev-step2}
   -\int_{B_r}(\divg A)(x\cdot\nabla u)
   &=-\int_{\partial B_r}(A\cdot\nu)(x\cdot\nabla u)\,dS
     +\int_{B_r}A\cdot\nabla(x\cdot\nabla u)\,dx.
\end{align}
Since
$$
   A\cdot\nabla(x\cdot\nabla u)
   =|\nabla u|^p+x\cdot\nabla\left(\frac{|\nabla u|^p}{p}\right),
$$
we get
\begin{equation}\label{eq:pohozaev-step3}
   \int_{B_r}A\cdot\nabla(x\cdot\nabla u)
   =\frac{p-n}{p}\int_{B_r}|\nabla u|^p\,dx
    +\frac r p\int_{\partial B_r}|\nabla u|^p\,dS.
\end{equation}
On the other hand,
\begin{equation}\label{eq:pohozaev-step4}
   \int_{B_r} f(u)x\cdot\nabla u\,dx
   =\int_{B_r}x\cdot\nabla F_f(u)\,dx
   =r\int_{\partial B_r}F_f(u)\,dS-n\int_{B_r}F_f(u)\,dx.
\end{equation}
Finally, testing the equation with $u$ gives
\begin{equation}\label{eq:pohozaev-step5}
   \int_{B_r}|\nabla u|^p\,dx
   =\int_{B_r}u f(u)\,dx+
     \int_{\partial B_r}u|\nabla u|^{p-2}u_\nu\,dS.
\end{equation}
Combining \eqref{eq:pohozaev-step1}--\eqref{eq:pohozaev-step5}, and using $x\cdot\nu=r$ and
$x\cdot\nabla u=ru_\nu$ on $\partial B_r$, gives \eqref{eq:pohozaev-identity}.
\end{proof}

Now we use the following consequence of the pole analysis, which gives a key observation in the current manuscript. The essential part is that a positive regular part of the solution forces the pole Pohozaev value to be negative.

\begin{prop}[Pohozaev value at a pole]\label{prop:pole-pohozaev}
Let $w$ be as in Theorem \ref{thm:KV}.  Suppose that
\begin{equation}\label{eq:pole-expansion-for-P}
     w(x)=\kappa |x|^{-\alpha}+\beta+o(1)
        \qquad\text{as }x\to0 , 
\end{equation}
with $\kappa>0$ and $\beta>0$, and suppose that the derivative asymptotics
from Theorem~\ref{thm:KV} hold.  Define, for $0<r<1$,
\begin{equation}\label{eq:p-harmonic-pohozaev-functional}
   \mathcal P_w(r)
   :=\int_{\partial B_r}
       \left[
          a\,w |\nabla w|^{p-2}w_\nu
          -\frac r p |\nabla w|^p
          +r|\nabla w|^{p-2}w_\nu^2
       \right]dS.
\end{equation}
Then $\mathcal P_w(r)$ is independent of $r$, and
\begin{equation*}\label{eq:pole-pohozaev-negative}
 \mathcal P_w(r) =  -a\,\omega_{n-1}(\alpha\kappa)^{p-1}\beta <0  \qquad\text{for every }0<r<1.
\end{equation*}
\end{prop}

\begin{proof}
First we prove that $\mathcal P_w(r)$ is independent of $r$. Let
$0<r_1<r_2<1$. Applying the Pohozaev identity Proposition~\ref{prop:pohozaev} in the annulus
$B_{r_2}\setminus \overline{B_{r_1}}$, where $w$ is $p$-harmonic, gives
$$
 P_w(r_2)-  P_w(r_1)=0.
$$
Hence $\mathcal P_w(r)=P_w(r)$ is constant.

It remains to compute this constant. Put $h(x):=\kappa |x|^{-\alpha}.$
Then on $\partial B_r$,
$$
  h_\nu=-\alpha\kappa r^{-\alpha-1},  \qquad |\nabla h|=\alpha\kappa r^{-\alpha-1}.
$$
The pure pole $h$ has zero Pohozaev value. Indeed,
$$
  h=\frac r\alpha |\nabla h|  \qquad\text{on }\partial B_r,
$$
and therefore
\begin{align*}
\mathcal P_h(r)
&= \int_{\partial B_r} \left[ a h|\nabla h|^{p-2}h_\nu
-\frac r p |\nabla h|^p +r|\nabla h|^{p-2}h_\nu^2 \right]\,dS           \\
&= \int_{\partial B_r} \left[ -a h|\nabla h|^{p-1}
+\frac{p-1}{p}r|\nabla h|^p \right]\,dS          \\
&= \int_{\partial B_r} \left[ -\frac a\alpha r|\nabla h|^p
+\frac{p-1}{p}r|\nabla h|^p \right]\,dS =0,
\end{align*}
since
$\frac a\alpha=\frac{p-1}{p}.$

Now we consider the model function $h+\beta$. Since $\nabla(h+\beta)=\nabla h$,
only the first term in the Pohozaev functional changes. Thus
$$
\begin{aligned}
\mathcal P_{h+\beta}(r)
&= \mathcal P_h(r) + a\beta \int_{\partial B_r} |\nabla h|^{p-2}h_\nu\,dS                         \\
&= a\beta \int_{\partial B_r} -(\alpha\kappa r^{-\alpha-1})^{p-1}\,dS            \\
&= -a\beta(\alpha\kappa)^{p-1} r^{n-1-(\alpha+1)(p-1)} \omega_{n-1}.
\end{aligned}
$$
Since $ (\alpha+1)(p-1)=n-1,$
we obtain
$$
\mathcal  P_{h+\beta}(r) = -a\,\omega_{n-1}(\alpha\kappa)^{p-1}\beta .
$$

We now justify the passage from the model function $h+\beta$ to the actual function $w$.  Write
$$
\eta(x):=w(x)-h(x)-\beta .
$$
By the pole expansion in Theorem~\ref{thm:KV} together with Corollary~\ref{cor:regular-part-gradient-decay}, 
$$   
\omega(r):=  \sup_{\partial B_r}|\eta| +  \sup_{\partial B_r} r|\nabla \eta|   \to  0  \qquad\text{as }r\to 0 .
$$
Set ${\mathcal C}_r:=\alpha\kappa r^{-\alpha-1}.$
Then on $\partial B_r$,
$$ 
\nabla h=-{\mathcal C}_r\nu,\qquad  h=\kappa r^{-\alpha}=\frac r\alpha {\mathcal C}_r .
$$
In particular, for $r$ small enough, $|\nabla w|$ is comparable to ${\mathcal C}_r$.

Let
$$ 
\mathcal I_r(u) := a u|\nabla u|^{p-2}u_\nu -\frac r p |\nabla u|^p +r|\nabla u|^{p-2}u_\nu^2
$$
denote the boundary integrand in the $p$-harmonic Pohozaev functional.
We claim that
$$
\int_{\partial B_r} \left| \mathcal I_r(w)-\mathcal I_r(h+\beta) \right|\,dS   \longrightarrow 0 .
$$

Indeed, since the vectors $\nabla w$ and $\nabla h$ both have size comparable to ${\mathcal C}_r$, the $C^1$-bounds for the maps
$$ 
\xi\mapsto |\xi|^{p-2}\xi\cdot\nu,\qquad  \xi\mapsto |\xi|^p,\qquad  \xi\mapsto |\xi|^{p-2}(\xi\cdot\nu)^2
$$
give, uniformly on $\partial B_r$,
\begin{align*}
\bigl| |\nabla w|^{p-2}w_\nu - |\nabla h|^{p-2}h_\nu
\bigr|  &\le C {\mathcal C}_r^{p-2}|\nabla\eta|,\\
\bigl| |\nabla w|^p-|\nabla h|^p \bigr| &\le C {\mathcal C}_r^{p-1}|\nabla\eta|,\\
\bigl| |\nabla w|^{p-2}w_\nu^2 - |\nabla h|^{p-2}h_\nu^2
\bigr| &\le C {\mathcal C}_r^{p-1}|\nabla\eta|.
\end{align*}
Therefore, 
$$
 \left| \mathcal I_r(w)-\mathcal I_r(h+\beta) \right| 
\le  C|\eta|{\mathcal C}_r^{p-1} + C(h+|\beta|){\mathcal C}_r^{p-2}|\nabla\eta|
+ Cr{\mathcal C}_r^{p-1}|\nabla\eta| .
$$
Using $h=(r/\alpha){\mathcal C}_r$, this becomes
$$ 
\left| \mathcal I_r(w)-\mathcal I_r(h+\beta) \right|
\le C|\eta|{\mathcal C}_r^{p-1} + Cr{\mathcal C}_r^{p-1}|\nabla\eta|
+ C|\beta|{\mathcal C}_r^{p-2}|\nabla\eta|.
$$
After integration over $\partial B_r$, and using
$$
 (\alpha+1)(p-1)=n-1,
$$
we obtain
\begin{align*}
\int_{\partial B_r} \left| \mathcal I_r(w)-\mathcal I_r(h+\beta) \right|\,dS
&\le C\sup_{\partial B_r}|\eta| + C\sup_{\partial B_r} r|\nabla\eta| + C r^\alpha \sup_{\partial B_r} r|\nabla\eta|  \\
&\le C\omega(r)\to0 .
\end{align*}
Hence
$$
 \mathcal  P_w(r)-\mathcal P_{h+\beta}(r)\to 0 \qquad\text{as }r\downarrow0 .
$$
Since we have already computed
$$
 \mathcal  P_{h+\beta}(r)  = -a\omega_{n-1}(\alpha\kappa)^{p-1}\beta ,
$$
it follows that
$$
 \lim_{r\downarrow0}\mathcal P_w(r)  =  -a\omega_{n-1}(\alpha\kappa)^{p-1}\beta .
$$
Because $\mathcal P_w(r)$ is independent of $r$, the same identity holds for every
$0<r<1$.
The right-hand side is strictly negative since $a>0$, $\kappa>0$, and
$\beta>0$.

\end{proof}

\subsection{The normalized Pohozaev neck lemma}

We now prove the key local compactness result.

\begin{rem}
Lemma~\ref{lem:neck} below is a conditional neck-improvement lemma. It assumes a preliminary singular-rate upper control
$$
   v_j(y)\lesssim |y|^{-a},\qquad a=\frac{n-p}{p},
$$
and upgrades it to the sharp $p$-harmonic fundamental rate
$$
   \inf_{\partial B_{\Gamma_j}}v_j\lesssim \Gamma_j^{-\alpha},
   \qquad \alpha=\frac{n-p}{p-1}.
$$
The proof does not use Kelvin inversion. The preliminary singular control gives a uniform annular Harnack inequality at the first-crossing scale; the Pohozaev sign argument then rules out slow decay at the fundamental rate.
\end{rem}








\begin{lem}[Pohozaev neck improvement under preliminary singular control]\label{lem:neck}
Let $\Gamma_j\to\infty$. For each $j$, let $v_j$ be a positive weak solution of
$$
   -\Delta_p v_j=f_j(v_j)\qquad\text{in }B_{4\Gamma_j}.
$$
Assume the following conditions.
\begin{enumerate}[(N1)]

\item Normalization and boundedness:
$$
   v_j(0)=1,
   \qquad
   0<v_j\le C_0\qquad\text{in }B_{4\Gamma_j}.
$$

\item Critical upper growth:
\begin{equation}\label{eq:fj-upper}
   0\le f_j(s)\le \Lambda_0 s^q\qquad\text{for all }s>0.
\end{equation}

\item Critical convergence:
\begin{equation}\label{eq:fj-convergence}
   f_j\to f_\infty
   \qquad\text{locally uniformly on }(0,\infty).
\end{equation}

\item Pohozaev sign: if
$$
   F_j(s):=\int_0^s f_j(t)\,dt,
$$
then
\begin{equation}\label{eq:fj-pohozaev-sign}
   nF_j(s)-a\,s f_j(s)\ge0
   \qquad\text{for every }s>0.
\end{equation}

\item Bubble limit:
$$
   v_j\to U\qquad\text{in }C^1_{\loc}(\mathbb R^n),
$$
where $U$ is an Aubin--Talenti $p$-bubble solving
$$
   -\Delta_p U=f_\infty(U)\qquad\text{in }\mathbb R^n.
$$

\item Preliminary singular upper control: There exist constants $C_s>0$ and
$R_s>1$, independent of $j$, such that
\begin{equation}\label{eq:prelim-singular-control}
   v_j(y)\le C_s |y|^{-a}
   \qquad\text{for }R_s\le |y|\le 4\Gamma_j.
\end{equation}

\end{enumerate}
Then there exists a constant
$$
   C=C(n,p,C_0,\Lambda_0,U,C_s,R_s)
$$
such that
\begin{equation}\label{eq:neck-conclusion}
   \inf_{\partial B_{\Gamma_j}}v_j\le C\Gamma_j^{-\alpha}
\end{equation}
for all sufficiently large $j$.
\end{lem}

\begin{proof}
Let $\kappa>0$ be defined by
\begin{equation*}\label{eq:kappa-def}
   \kappa:=\lim_{|x|\to\infty}|x|^\alpha U(x).
\end{equation*} 
Since $U$ is an Aubin--Talenti bubble, its explicit profile satisfies
$$
-\Delta_p U=L_U U^q
\qquad\text{in }\R^n
$$
for some constant $L_U>0$. Since also
$$
-\Delta_p U=f_\infty(U),
$$
it follows that
$$
f_\infty(U)=L_U U^q
\qquad\text{in }\R^n .
$$

We argue by contradiction.  Suppose that \eqref{eq:neck-conclusion} fails.  Then, after passing to a
subsequence,
\begin{equation}\label{eq:slow-decay}
   \Gamma_j^\alpha\inf_{\partial B_{\Gamma_j}}v_j\to\infty.
\end{equation}

For $r>0$, define the closed annulus
$$
A_r:=\{x\in\mathbb R^n:\ r/2\le |x|\le 2r\}
$$
and the annular height
$$
H_j(r):=r^\alpha\sup_{A_r}v_j .
$$
We also write
$$
H_U(r):=r^\alpha\sup_{A_r}U .
$$
Since $U$ is an Aubin--Talenti bubble and
$$
\lim_{|x|\to\infty}|x|^\alpha U(x)=\kappa ,
$$
we may choose $R_0>2$ so large that
$$
S_U:=\sup_{r\ge R_0}H_U(r)<\infty .
$$
By the local uniform convergence $v_j\to U$, we have
$H_j(R_0)\to H_U(R_0).$
Hence, after increasing $j$ if necessary,
$$
H_j(R_0)\le S_U+1 .
$$

Set
$$
   \mathcal D:=\left\{x\in\mathbb R^n:\frac13<|x|<3\right\},
   \qquad
   A_{\rm in}:=\left\{x\in\mathbb R^n:\frac34\le |x|\le \frac32\right\}.
$$
The preliminary singular control will imply below that, on $\mathcal D$, the rescaled
functions $w_j$ solve an equation of the form
$$
   -\Delta_p w_j=b_j(x)w_j^{p-1},
   \qquad
   0\le b_j\le B_s,
$$
with $B_s$ independent of $j$. Let $C_{\rm Har}$ be the Harnack-chain constant in
$\mathcal D$ for such equations, chosen so that
$$
   \sup_{A_1}h\le C_{\rm Har}\inf_{A_{\rm in}}h,
   \qquad A_1:=B_2\setminus B_{1/2},
$$
for every positive weak solution of
$$
   -\Delta_p h=b(x)h^{p-1},\qquad 0\le b\le B_s,
$$
in $\mathcal D$. Set $c_H:=C_{\rm Har}^{-1}$. Choose $K$ so large that
$$
K>\max\left\{S_U+1,\ 2S_U,\ \frac{\kappa+2}{c_H}\right\}.
$$

Then
\begin{equation*}\label{eq:Hj-R0-bound}
H_j(R_0)<K
\end{equation*}
for all sufficiently large $j$.

By the contradiction assumption \eqref{eq:slow-decay},
\begin{equation*}\label{eq:Hj-large-outer}
H_j(\Gamma_j/2)
\ge (\Gamma_j/2)^\alpha \inf_{\partial B_{\Gamma_j}}v_j \to\infty .
\end{equation*}
Therefore, by the continuity of $r\mapsto H_j(r)$, there exists a first radius
$$
r_j\in [R_0,\Gamma_j/2]
$$
such that
\begin{equation}\label{eq:first-crossing}
H_j(r_j)=K,\qquad H_j(r)\le K\quad\text{for }R_0\le r\le r_j .
\end{equation}

We claim that
$$
r_j\to\infty .
$$
Indeed, if not, then after passing to a subsequence we would have
$$
r_j\to r_\ast\in [R_0,\infty).
$$
Since the annuli $A_{r_j}$ remain in a fixed compact subset of
$\mathbb R^n\setminus\{0\}$, the local uniform convergence $v_j\to U$, together with the continuity of $U$, gives
$$
H_j(r_j)\to H_U(r_\ast).
$$
But $r_\ast\ge R_0$, so 
$$
H_U(r_\ast)\le S_U<K,
$$
contradicting $H_j(r_j)=K$ in \eqref{eq:first-crossing}. Hence $r_j\to\infty$.

The first-crossing property implies the pointwise estimate
\begin{equation}\label{eq:first-crossing-pointwise}
   v_j(y)\le C K |y|^{-\alpha}
   \qquad\hbox{for }2R_0\le |y|\le 2r_j.
\end{equation}
Indeed, if $|y|\le r_j$, choose $r=|y|$ in \eqref{eq:first-crossing}; if $r_j<|y|\le2r_j$, use the annulus
$A_{r_j}$.

Define the neck rescaling
\begin{equation}\label{eq:wj-def}
   w_j(x):=r_j^\alpha v_j(r_jx).
\end{equation}
Then $w_j$ is defined at least in $B_2$, and
\begin{equation}\label{eq:wj-annulus-height}
   \sup_{A_1}w_j=K,
   \qquad A_1=B_2\setminus B_{1/2}.
\end{equation}
The equation for $w_j$ is
\begin{equation}\label{eq:wj-eqn}
   -\Delta_p w_j=\widetilde f_j(w_j),
\end{equation}
where
\begin{equation}\label{eq:ftilde-def}
   \widetilde f_j(s):=r_j^n f_j(r_j^{-\alpha}s).
\end{equation}
To verify \eqref{eq:ftilde-def}, note that
$$
   \nabla w_j(x)=r_j^{\alpha+1}\nabla v_j(r_jx)
$$
and hence
$$
   -\Delta_p w_j(x)=r_j^{\alpha(p-1)+p}f_j(v_j(r_jx)).
$$
Since $\alpha(p-1)=n-p$, this exponent is $n$.

Using \eqref{eq:fj-upper}, we get
\begin{equation}\label{eq:ftilde-bound}
   0\le \widetilde f_j(s)
   \le \Lambda_0 r_j^{n-\alpha q}s^q.
\end{equation}
The exponent is negative, since
\begin{equation}\label{eq:critical-exponent-neck}
   n-\alpha q
   =n-\frac{n-p}{p-1}\cdot\frac{n(p-1)+p}{n-p}
   =-\frac{p}{p-1}.
\end{equation}

We now prove the   annular Harnack estimate. Since $r_j\to\infty$, for all
large $j$ we have $r_j/3\ge R_s$. Also, because $r_j\le \Gamma_j/2$, one has
$3r_j<4\Gamma_j$. Hence \eqref{eq:prelim-singular-control} applies to $r_jx$ for
every $x\in\mathcal D$. Therefore
$$
   w_j(x)
   =
   r_j^\alpha v_j(r_jx)
   \le
   C_s r_j^{\alpha-a}|x|^{-a}
   \qquad\text{for }x\in\mathcal D.
$$
Since $w_j>0$, write
$$
   -\Delta_p w_j=b_j(x)w_j^{p-1},
   \qquad
   b_j(x):=\frac{\widetilde f_j(w_j(x))}{w_j(x)^{p-1}}.
$$
Using \eqref{eq:ftilde-bound} and the preceding estimate, we obtain on $\mathcal D$
$$
   0\le b_j(x)
   \le
   \Lambda_0 r_j^{-\frac{p}{p-1}}w_j(x)^{q-p+1}
   \le
   \Lambda_0 C_s^{q-p+1}|x|^{-a(q-p+1)}
   \le B_s,
$$
because
$$
   (\alpha-a)(q-p+1)=\frac{p}{p-1}.
$$
Thus the Harnack-chain inequality in $\mathcal D$ gives
\begin{equation}\label{eq:buffered-harnack-wj}
   \inf_{A_{\rm in}}w_j
   \ge c_H\sup_{A_1}w_j
   =c_HK.
\end{equation}
On every compact set $K_0\Subset B_2\setminus\{0\}$, the first-crossing estimate
\eqref{eq:first-crossing-pointwise} gives a uniform upper bound for $w_j$. Therefore
\eqref{eq:ftilde-bound} and \eqref{eq:critical-exponent-neck} imply
$$
   \widetilde f_j(w_j)\to0
   \qquad\text{locally uniformly on compact subsets of }B_2\setminus\{0\}.
$$

By the local Harnack inequality and local $C^{1,\beta}$ estimates for $p$-Laplace type equations, after
passing to a subsequence,
\begin{equation*}\label{eq:wj-to-w}
   w_j\to w\qquad\hbox{in }C^1_{\loc}(B_2\setminus\{0\}),
\end{equation*}
where
\begin{equation*}\label{eq:w-pharmonic}
   -\Delta_p w=0\qquad\hbox{in }B_2\setminus\{0\}.
\end{equation*}
The relevant regularity and Harnack theory for quasilinear equations goes back to Serrin, Trudinger,
Ural'tseva, Lewis, Tolksdorf, and others; see~\cite{S1964,T1967,L1983,T1984,PS2007}.

The estimate \eqref{eq:first-crossing-pointwise} gives
\begin{equation*}\label{eq:w-growth-upper}
   0<w(x)\le C K |x|^{-\alpha}
   \qquad\hbox{in }B_2\setminus\{0\}.
\end{equation*}
Thus, we can apply the Kichenassamy--V\'eron pole theorem, namely Theorem~\ref{thm:KV}.

We identify the coefficient of the pole.  For $0<\rho<1$, integrating \eqref{eq:wj-eqn} over $B_\rho$ yields
\begin{equation}\label{eq:flux-wj}
   -\int_{\partial B_\rho}|\nabla w_j|^{p-2}\partial_\nu w_j\,dS
   =\int_{B_\rho}\widetilde f_j(w_j)\,dx.
\end{equation}
Moreover, \eqref{eq:wj-def} and  \eqref{eq:ftilde-def} give
\begin{equation}\label{eq:undo-flux-scaling}
   \int_{B_\rho}\widetilde f_j(w_j)\,dx
   =\int_{B_{\rho r_j}} f_j(v_j)\,dy.
\end{equation}
 Let $R>2R_0$ be fixed. The convergence $v_j\to U$ and
\eqref{eq:fj-convergence} imply
$$
   \int_{B_R}f_j(v_j)\,dy
   \to
   \int_{B_R}L_U U^{p^*-1}\,dy.
$$
For $R\le |y|\le \rho r_j$, the first-crossing estimate
\eqref{eq:first-crossing-pointwise} gives
$$
   v_j(y)\le CK |y|^{-\alpha}.
$$
Therefore, by \eqref{eq:fj-upper},
$$
   f_j(v_j(y))\le C |y|^{-\alpha q}.
$$
Since
$$
   \alpha q=n+\frac{p}{p-1}>n,
$$
the right-hand side is integrable at infinity. Hence, after first letting $j\to\infty$
and then $R\to\infty$,
\begin{equation*}\label{eq:source-total-limit}
   \int_{B_{\rho r_j}} f_j(v_j)\,dy
   \to\int_{\R^n}L_U U^{p^*-1}\,dy.
\end{equation*}
Passing to the limit in \eqref{eq:flux-wj}, we obtain
\begin{equation*}\label{eq:w-flux-total}
   -\int_{\partial B_\rho}|\nabla w|^{p-2}w_\nu\,dS
   =\int_{\R^n}L_U U^{p^*-1}\,dy.
\end{equation*}
The bubble tail \eqref{eq:bubble-tail} gives
\begin{equation*}\label{eq:bubble-flux}
   \int_{\R^n}L_U U^{p^*-1}\,dy
   =\omega_{n-1}(\alpha\kappa)^{p-1}.
\end{equation*}
Therefore the pole coefficient of $w$ is exactly $\kappa$, and
\begin{equation*}\label{eq:w-kappa-plus-z}
   w(x)=\kappa |x|^{-\alpha}+z(x),
   \qquad z\in L^\infty_{\loc}(B_2).
\end{equation*}

We next prove that the regular part is positive. By the buffered annular Harnack
estimate \eqref{eq:buffered-harnack-wj},
$$
   \inf_{A_{\rm in}}w_j\ge c_HK.
$$
Since $A_{\rm in}\Subset B_2\setminus\{0\}$, the convergence
$w_j\to w$ in $C^1_{\loc}(B_2\setminus\{0\})$ gives
$$
   \inf_{A_{\rm in}}w\ge c_HK.
$$
In particular, because $\partial B_1\subset A_{\rm in}$, the choice of $K$ gives
\begin{equation}\label{eq:w-boundary-positive-regular}
   w\ge \kappa+2
   \qquad\text{on }\partial B_1 .
\end{equation}

For $0<\delta<\kappa$, define
\begin{equation}\label{eq:phi-delta}
   \phi_\delta(x):=(\kappa-\delta)|x|^{-\alpha}+1.
\end{equation}
Both $w$ and $\phi_\delta$ are $p$-harmonic in $B_1\setminus\{0\}$. By
\eqref{eq:w-boundary-positive-regular}, $w\ge\phi_\delta$ on $\partial B_1$.
Near $0$, using \eqref{eq:w-kappa-plus-z},
$$
   w-\phi_\delta
   =
   \delta |x|^{-\alpha}+z(x)-1>0
$$
for $|x|$ sufficiently small. The comparison principle in $B_1\setminus B_\varepsilon$,
followed by letting $\varepsilon\downarrow0$, gives
$$
   w\ge\phi_\delta
   \qquad\text{in }B_1\setminus\{0\}.
$$
Letting $\delta\downarrow0$, we obtain
$$
   z(x)=w(x)-\kappa |x|^{-\alpha}\ge1
   \qquad\text{in }B_1\setminus\{0\}.
$$
In particular,
$$
   \beta:=\lim_{x\to0}z(x)\ge1,
$$
and the regular part is positive at the pole in the sense of
Proposition~\ref{prop:pole-pohozaev}.

Since $w$ is $p$-harmonic in $B_2\setminus\{0\}$, Proposition \ref{prop:pole-pohozaev} gives
\begin{equation*}\label{eq:w-P-negative}
   \mathcal P_w(1)=  -a\,\omega_{n-1}(\alpha\kappa)^{p-1}\beta<0.
\end{equation*}

On the other hand, the Pohozaev sign for $w_j$ gives the opposite inequality.  Indeed, define
$$
   \widetilde F_j(s):=\int_0^s\widetilde f_j(t)\,dt.
$$
The sign \eqref{eq:fj-pohozaev-sign} is preserved by the scaling \eqref{eq:ftilde-def}; hence
$$
   n\widetilde F_j(s)-a\,s\widetilde f_j(s)\ge0.
$$
By Proposition \ref{prop:pohozaev},
\begin{equation*}\label{eq:P-wj-nonnegative}
   P_{w_j}(1)
   =\int_{B_1}\left[n\widetilde F_j(w_j)-a\,w_j\widetilde f_j(w_j)\right]dx
   \ge0.
\end{equation*}
On $\partial B_1$, $w_j\to w$ in $C^1$, and by \eqref{eq:ftilde-bound}--\eqref{eq:critical-exponent-neck},
$$
   \widetilde F_j(w_j)\to0
   \qquad\hbox{uniformly on }\partial B_1.
$$
Thus
\begin{equation*}\label{eq:P-limit-nonnegative}
   \mathcal P_w(1)=\lim_{j\to\infty}P_{w_j}(1)\ge0.
\end{equation*}
This contradicts Proposition~\ref{prop:pole-pohozaev}.  Therefore \eqref{eq:slow-decay} is impossible, and the proof of
\eqref{eq:neck-conclusion} is complete.
\end{proof}

\section{The Harnack theorem}

With the neck lemma in hand, we now give the proof of Theorem~\ref{thm:main-intro}, which was stated in the introduction.
\begin{proof}[Proof of Theorem~\ref{thm:main-intro}]
Assume by contradiction that the estimate fails.  Then there exist
sequences $R_j>0$ and positive weak solutions $u_j$ of $-\Delta_p u_j=g(u_j)$
in $B_{3R_j}$ such that
\begin{equation}\label{eq:contra-assumption}
\sup_{B_{R_j}}u_j\ge 1,\qquad
R_j^{\,n-p}\,M_j\,m_j^{p-1}\to +\infty,
\end{equation}
where we set
$$M_j:=\sup_{B_{R_j}}u_j,\qquad m_j:=\inf_{B_{2R_j}}u_j . $$

Define 
$d(x):=\operatorname{dist}(x,\partial B_{2R_j})$ for $x\in B_{2R_j}$.
Choose $x_j\in B_{2R_j}$ such that
\begin{equation}\label{eq:doubling-selection}
d(x_j)^{n-p}u_j(x_j)=\max_{B_{2R_j}} d(x)^{n-p}u_j(x).
\end{equation}
Set
$$\widetilde M_j:=u_j(x_j),\quad d_j:=d(x_j),\quad \sigma_j:=\frac{d_j}{8}.$$
Since $d(x)\ge R_j$ for $x\in B_{R_j}$, \eqref{eq:doubling-selection} yields
\begin{equation}\label{eq:Mj-lower-bound}
\widetilde M_j\, d_j^{\,n-p}\ge R_j^{\,n-p} M_j .
\end{equation}

If $y\in B_{4\sigma_j}(x_j)=B_{d_j/2}(x_j)$, then $d(y)\ge d_j/2$.  Hence, by
\eqref{eq:doubling-selection},
\begin{equation*}\label{eq:local-height-bound}
   u_j(y)\le 2^{n-p}\widetilde M_j
   \qquad\hbox{for }y\in B_{4\sigma_j}(x_j).
\end{equation*}
Combining \eqref{eq:contra-assumption} with \eqref{eq:Mj-lower-bound}, we also obtain
\begin{equation}\label{eq:Mtilde-sigma-large-2}
   \widetilde M_j\sigma_j^{n-p}m_j^{p-1}\to\infty.
\end{equation}

Define the critical blow-up scale
\begin{equation*}\label{eq:rho-Gamma-def}
   \rho_j:=\widetilde M_j^{-\frac{p}{n-p}},\qquad
\Gamma_j:=\frac{\sigma_j}{\rho_j}
        =\sigma_j\,\widetilde M_j^{\frac{p}{n-p}} .
\end{equation*}
Then $\Gamma_j\to\infty$.  Indeed,
$$
   \Gamma_j^{n-p}
   =\sigma_j^{n-p}\widetilde M_j^p
   =\left(\widetilde M_j\sigma_j^{n-p}m_j^{p-1}\right)
     \left(\frac{\widetilde M_j}{m_j}\right)^{p-1},
$$
and the first factor tends to infinity by \eqref{eq:Mtilde-sigma-large-2}, while the second is at least $1$.

  Now, define the rescaled functions
\begin{equation}\label{eq:vj-def}
v_j(y):=\widetilde M_j^{-1}\,u_j(x_j+\rho_j y),\qquad |y|<4\Gamma_j .
\end{equation}
Then $$v_j(0)=1, \quad 0<v_j(y)\le 2^{n-p}\quad \text{in}\ B_{4 \Gamma_j}.$$ Moreover, 
\begin{equation}\label{eq:vj-eqn}
-\Delta_p v_j = f_j(v_j), 
\end{equation}
where $$f_j(s):=\widetilde M_j^{-q}\,g(\widetilde M_j s).$$  The global upper bound \eqref{eq:g-upper-critical} gives
$$
   f_j(s)\le \Lambda s^q.
$$

Passing to a subsequence we may assume that
$A:=\lim_{j\to\infty}\widetilde M_j\in[2^{p-n},\infty]$. Note that $A\ge 2^{-(n-p)}>0$ because $M_j\ge 1$ and
$\widetilde M_j\ge 2^{-(n-p)}M_j$ by \eqref{eq:Mj-lower-bound} and $d_j\le 2R_j$. By local compactness and Harnack, after passing to a subsequence,
\begin{equation*}\label{eq:vj-to-V-final}
   v_j\to V
   \qquad\hbox{in }C^1_{\loc}(\R^n),
\end{equation*}
where
\begin{equation*}\label{eq:V-entire-final}
   -\Delta_p V=g_A(V)\qquad\hbox{in }\R^n,
\end{equation*}
with
$$
   0<V\le 2^{n-p},
   \qquad
   V(0)=1.
$$ Here the nonlinearity $g_A(s)$ is defined by \eqref{def-ga}, with the constant $L$ taken from \eqref{eq:g-limit}.

By the classification hypothesis stated in Assumption~\ref{ass:classification},
$V$ is an Aubin--Talenti $p$-bubble, and substituting its explicit form into the
equation yields $g_A(s)=L_A s^q$ on the range $0<s\le \max_{\mathbb R^n} V$.

Note that
$$
   f_j(s)\to  g_A(s)
   \qquad\hbox{locally uniformly for }s>0. 
$$
The monotonicity of $s^{-q}g(s)$ gives the Pohozaev sign for $f_j$ by
Lemma~\ref{lem:pohozaev-sign}. By the second assumption in Assumption~\ref{ass:classification}, this normalized
blow-up neck sequence also satisfies the preliminary singular control \eqref{eq:prelim-singular-control}. Therefore Lemma~\ref{lem:neck} applies and yields
\begin{equation}\label{eq:neck-conclusion-main}
\inf_{\partial B_{\Gamma_j}} v_j \le C \,\Gamma_j^{-\alpha}.
\end{equation}

For $|y|=\Gamma_j$, the point $x_j+\rho_j y\in B_{\sigma_j}(x_j)\subset B_{2R_j}$.
Consequently, 
\begin{equation}\label{eq:boundary-lower-mj}v_j(y)=\widetilde M_j^{-1}\,u_j(x_j+\rho_j y)\ge \widetilde M_j^{-1}\,m_j.\end{equation}
Combining this with~\eqref{eq:neck-conclusion-main} gives
$$  
\frac{m_j}{\widetilde M_j}\le C \,\Gamma_j^{-\alpha}.
$$
Using $\Gamma_j=\sigma_j\,\widetilde M_j^{p/(n-p)}$ and
$\alpha=(n-p)/(p-1)$,  we get
$$m_j\le C \,\sigma_j^{-\alpha}\,\widetilde M_j^{-\frac{1}{p-1}}.$$
Therefore 
 $$\widetilde M_j \sigma_j^{n-p} \,m_j^{p-1}\le C,$$ which contradicts~\eqref{eq:Mtilde-sigma-large-2}.
The contradiction shows that the original assumption was false, and the Harnack
inequality
holds for every $R>0$ and every positive weak solution $u$ in $B_{3R}$
with $\sup_{B_R}u\ge 1$.  The proof is complete.
\end{proof}

\end{document}